\documentclass[letterpaper, 10pt, conference]{ieeeconf}      % Use this line for a4 paper

\usepackage{umlaute,eurosym,comment}
\usepackage{amsmath,amssymb,bm,upgreek,theorem}
\usepackage{float,epsfig,pstricks,pst-node,auto-pst-pdf,subfig,xcolor}
\usepackage{wrapfig}
\usepackage{ifpdf,cite}
\usepackage{multirow}
\usepackage{flushend}
\usepackage{tikz,pgfplots}
\psset{unit=1mm,framesep=10pt,fillstyle=none}  % pstricks basic settings
\degrees[360]                                  % segmentation of the unit circle
\psset{linewidth=0.9pt}                        % standard line width
\newlength\figureheight	
\newlength\figurewidth	

\newtheorem{theorem}{Theorem}

\newtheorem{lemma}[theorem]{Lemma}

\newtheorem{proposition}[theorem]{Proposition}
\newtheorem{definition}[theorem]{Definition}
\newtheorem{remark}[theorem]{Remark}
\newtheorem{assumption}[theorem]{Assumption}
\providecommand{\ef}{\;.}
\providecommand{\ec}{\;,}

\providecommand{\fa}{\forall\;}         % for all
\providecommand{\eqdef}{\triangleq}     % equals by definition
\providecommand{\eqcon}{\overset{!}{=}} % equals by condition
\providecommand{\tp}{^{\scriptsize{\mathrm{T}}}} % transpose
\providecommand{\to}{\rightarrow}                        % converges to
\newcommand{\conv}[1]{\text{conv}\left(#1\right)}        % convex hull
\newcommand{\vct}[1]{\mathrm{vec}\left(#1\right)}  % vectorization

\newcommand{\BZ}{\mathbb{Z}}        % integers
\newcommand{\BR}{\mathbb{R}}        % real numbers
\newcommand{\BX}{\mathbb{X}} 
\newcommand{\BU}{\mathbb{U}}
\newcommand{\BD}{\mathbb{D}} 
\newcommand{\CX}{\mathcal{X}} 
\newcommand{\CP}{\mathcal{P}} 

\definecolor{ULOcean}{RGB}{0,75,90}
\definecolor{ULOrange}{RGB}{236,116,4}
\definecolor{LiteRed}{RGB}{255,206,206}
\definecolor{MedRed}{RGB}{204,24,24}
\definecolor{DarkRed}{RGB}{127,15,15}
\definecolor{LiteYellow}{RGB}{255,255,100}
\definecolor{MedYellow}{RGB}{255,255,25}
\definecolor{DarkYellow}{RGB}{225,225,0}
\definecolor{LiteBlue}{RGB}{170,184,255}
\definecolor{MedBlue}{RGB}{10,20,204}
\definecolor{DarkBlue}{RGB}{27,29,120}
\definecolor{LiteGreen}{RGB}{24,204,75}
\definecolor{MedGreen}{RGB}{0,130,37}
\definecolor{DarkGreen}{RGB}{0,61,17}
\definecolor{LiteViolet}{RGB}{110,0,225}
\definecolor{MedViolet}{RGB}{73,0,148}
\definecolor{DarkViolet}{RGB}{53,0,107}
\definecolor{LiteOrange}{RGB}{255,226,117}
\definecolor{MedOrange}{RGB}{255,206,0}
\definecolor{DarkOrange}{RGB}{204,167,10}
\definecolor{LiteGray}{RGB}{230,230,230}
\definecolor{MedGray}{RGB}{160,160,160}
\definecolor{DarkGray}{RGB}{100,100,100}

\title{\LARGE \bf
Robust MPC for Large-scale Linear Systems
}

\author{Georg Schildbach$^{*}$
\thanks{$^{*}$Autonomous Systems Laboratory, University of L{\"u}beck, L{\"u}beck, Germany (email: {\tt\small georg.schildbach@uni-luebeck.de}).}
}

\begin{document}

\maketitle
\thispagestyle{empty}
\pagestyle{empty}

%---------------------------------------------------------------------------------------------------

\begin{abstract}
  State-of-the-art approaches of Robust Model Predictive Control (MPC) are restricted to linear systems of relatively small scale, i.e., with no more than about 5 states. The main reason is the computational burden of determining a robust positively invariant (RPI) set, whose complexity suffers from the curse of dimensionality. The recently proposed approach of Deadbeat Robust Model Predictive Control (DRMPC) is the first that does not rely on an RPI set. Yet it comes with the full set of essential system theoretic guarantees. DRMPC is hence a viable option, in particular, for large-scale systems. This paper introduces a detailed design procedure for DRMPC. It is shown that the optimal control problem generated for DRMPC has exactly the same computational complexity as Nominal MPC. A numerical study validates its applicability to randomly generated large-scale linear systems of various dimensions.
\end{abstract}

%---------------------------------------------------------------------------------------------------
\vspace*{-0.05cm}
\section{INTRODUCTION}\label{Sec:Intro}

Originating from process industry, for systems with slow dynamics \cite{MorariLee:1999}, Model Predictive Control (MPC) has spread to applications with sampling times in the millisecond range, and even below \cite{VasquezEtAl:2014}. However, MPC remains attractive for the control of large-scale systems, such as chemical plants \cite{VenkatEtAl:2007}, building climate control \cite{RobillartEtAl:2019}, energy networks \cite{JiaEtAl:2024}, and supply chains \cite{SchildMor:2016}.

In most practical applications, an exact model of the system is not available. Therefore, a certain degree of robustness to model uncertainty or disturbances is of great importance. The theory of Robust MPC has advanced significantly over the past decades, with powerful and highly researched frameworks being readily available \cite{Kerrigan:2000,Mayne:2005,Goulart:2006}. However, all existing methods of Robust MPC are limited to small-scale systems. The main reason is that all of them require the (offline) computation of robust positively invariant sets, including the most popular approaches, Tube MPC \cite{Mayne:2005} and Robust MPC with Affine Disturbance Feedback \cite{Goulart:2006}. As a result, these conventional approaches may work for systems with state dimensions up to 5, as a rule of thumb.

Commonly explored techniques to circumvent this limitations include model reduction \cite{RobillartEtAl:2019} and distributed MPC \cite{VenkatEtAl:2007}. However, model reduction is not always applicable and leads to additional conservatism, because part of the system dynamics are modeled as disturbances. Distributed MPC comes with an increased computational burden, even if it is distributed, and also leads additional conservatism.

This paper leverages the recently proposed approach of Deadbeat Robust Model Predictive Control (DRMPC) \cite{Schildbach:2025,SchildbachAbbas:2025} for large-scale systems. Roughly speaking, the target is systems with state dimensions of 5 or higher. A precise definition of what `large-scale' means in this context, and which systems fall within the scope of DRMPC, is given in the paper. The paper also provides a detailed description of all steps required for the practical design of the DRMPC, along with a summary of essential system theoretic guarantees. A numerical study validates the applicability of DRMPC to randomly generated large-scale systems, showing that it goes far beyond the capabilities of conventional Robust MPC approaches.

\section{TERMINOLOGY AND NOTATION}

$\mathbb{R}$ and $\mathbb{Z}$ denote the sets of real and integral numbers. $\mathbb{R}_{+}$ ($\mathbb{R}_{0+}$) and $\mathbb{Z}_{+}$ ($\mathbb{Z}_{0+}$) are then the sets of positive (non-negative) real and integral numbers, respectively. 

$\mathbb{R}^{n}$ represents a vector space of dimension $n\in\mathbb{Z}_{+}$. The $i$-th component of a vector $v\in\mathbb{R}^{n}$ is denoted $v_{[i]}\in\BR$. The notation $\|\cdot\|$ is used to denote any norm. If $A\in\BR^{n\times m}$ is a matrix, $A_{[i,\cdot]}\tp\in\BR^{m}$ denotes the $i$-th row and $A_{[\cdot,j]}\in\BR^{n}$ the $j$-th column of $A$. The expression $\vct{A}\in\BR^{nm}$ represents the vectorization of $A$, i.e., a column vector containing the stacked-up columns of $A$.

If $\mathbb{S},\mathbb{T}\subset\mathbb{R}^{n}$ are sets, the \emph{Minkowski sum} is defined as
\begin{equation*}
	\mathbb{S}\oplus \mathbb{T}\triangleq\bigl\{s+t\in\mathbb{R}^{n}\:\big|\: s\in \mathbb{S},\:t\in \mathbb{T}\}\ec
\end{equation*}
and the \emph{Pontryagin difference} as
\begin{equation*}
	\mathbb{S}\ominus \mathbb{T}\triangleq\bigl\{\xi\in\mathbb{R}^{n}\:\big|\:\xi+t\in \mathbb{S}\:\:\forall\:t\in \mathbb{T}\}\;.
\end{equation*}

A polyhedron $\mathbb{P}$ is the finite intersection of closed half-spaces in $\mathbb{R}^{n}$, and a polytope is a bounded polyhedron. The \emph{H-representation} $\mathcal{H}(\mathbb{P})$ is the (unique) finite set of non-redundant half-spaces, called \emph{facets}, constituting $\mathbb{P}$, i.e., some $G\in\BR^{m\times n}$ and $h\in\BR^{m}$ such that
\begin{equation*}
	\mathbb{P}=\bigl\{\xi\in\BR^{n}\:\big|\:G_{i,\cdot}\xi\leq h_{i}\:\:\forall\:i=1,2,\dots,m\}\ef
\end{equation*}
The \emph{V-representation} $\mathcal{V}(\mathbb{P})$ is the unique, finite set of non-redundant points $v_{j}$, called \emph{vertices}, such that
\begin{equation*}
	\mathbb{P}=\conv{\bigl\{v_{j}\in\BR^{n}\:\big|\:j=1,2,\dots,p\bigr\}}\ef
\end{equation*}
The expression $\conv{\cdot}$ denotes the convex hull of a set. The procedure of finding $\mathcal{V}(\mathbb{P})$ given $\mathcal{H}(\mathbb{P})$ is called \emph{vertex enumeration}; and finding $\mathcal{H}(\mathbb{P})$ given $\mathcal{V}(\mathbb{P})$ is called \emph{facet enumeration}. 

A function $\alpha:\BR_{0+}\rightarrow\BR_{0+}$ is a \emph{$K$-function} if it is continuous, strictly monotonically increasing, and $\alpha(0)=0$. It is a \emph{$K_{\infty}$-function} if, in addition, $\alpha(r)\rightarrow\infty$ as $r\rightarrow\infty$. A function $\beta:\BR_{0+}\times\BZ_{0+}\rightarrow\BR_{0+}$ is a \emph{$KL$-function} if $\beta(\,\cdot\,,k)$ is a $\text{K}$-function for any fixed $k\in\BZ_{0+}$, and $\beta(r,\,\cdot\,)$ is monotonically decreasing with $\beta(r,k)\rightarrow 0$ as $k\rightarrow\infty$ for any fixed $r\in\BR_{0+}$.

\section{PROBLEM STATEMENT}

\subsection{Control System}\label{Sec:System}

Starting point is a (large-scale) discrete-time, linear-time invariant control system with additive disturbances
\begin{equation}\label{Equ:DTSystem}
  x_{k+1}=Ax_{k}+Bu_{k}+d_{k}\;.
\end{equation}
Here $k\in\mathbb{Z}_{0+}$ denotes the time step, $x_{k}\in\mathbb{R}^{n}$ and $u_{k}\in\mathbb{R}^{m}$ are the state and the input and $d_{k}\in\mathbb{R}^{n}$ is the disturbance at time step $k$, respectively. System \eqref{Equ:DTSystem} comes with a given initial condition $x_{0}\in\mathbb{R}^{n}$.

\begin{assumption}[\textbf{control system}]\label{The:System}
  (a) The state of the system $x_{k}$ can be measured in each step $k$. (b) The system matrix $A\in\BR^{n\times n}$ and input matrix $B\in\BR^{n\times m}$ form a controllable pair.
\end{assumption}

The control objective is to regulate the state of system \eqref{Equ:DTSystem} to the origin, while respecting state and input constraints:
\begin{equation}\label{Equ:Constraints}
  x_{k}\in\BX\subseteq\BR^{n}\;,\enspace u_{k}\in\BU\subseteq\BR^{m}\quad\fa k\in\mathbb{Z}_{0+}\ef
\end{equation}
The additive disturbance is also contained in a given disturbance set:
\begin{equation}\label{Equ:DisSet}
  d_{k}\in\BD\subseteq\BR^{n}\quad\fa k\in\mathbb{Z}_{0+}\;.
\end{equation}

\begin{assumption}[\textbf{constraint sets}]\label{The:Constraints}
  (a) The input and state constraint sets $\BU$ and $\BX$ are polyhedrons containing the origin. (b) The disturbance set $\BD$ is a polytope containing the origin.
\end{assumption}

\subsection{Large-Scale System}\label{Sec:LargeScale}

The theory of this paper holds for \emph{all} linear systems \eqref{Equ:DTSystem}. However, the specific focus of this paper is on \emph{large-scale} systems. This section provides a more precise mathematical characterization of what `large-scale' means in the context or Robust MPC; i.e., what is assumed to be computationally possible and what is prohibitive. It is assumed that a desired prediction horizon $N\geq n$ has already been selected.

\begin{assumption}[\textbf{Nominal MPC}]\label{The:NominalMPC}
  (a) The sets $\BU$, $\BX$, $\BD$ are given, and stored, in H-representation. (b) The sets $\BU$ and $\BX$ are such that the finite-horizon optimal control problem (FHOCP) corresponding to \emph{Nominal MPC}, i.e., without any disturbances, is computationally feasible (in real time).
\end{assumption}

The H-representations of $\BU$ and $\BX$ and $\BD$ are denoted as $G^{(\mathrm{u})}$, $h^{(\mathrm{u})}$ and $G^{(\mathrm{x})}$, $h^{(\mathrm{x})}$ and $G^{(\mathrm{d})}$, $h^{(\mathrm{d})}$, respectively. The conversion of an H-representation to its corresponding V-representation, also called \emph{vertex enumeration}, is known to be computationally hard. In fact, the complexity of vertex enumeration is usually polynomial in the input and output size, and exponential only in some special (degenerate) cases \cite{AvisEtAl:1997}. The input and output size, i.e., the combination of H-representation and V-representation of a polyhedron, however, generally grows exponentially with the dimension, even in simple cases (e.g., a hyper-cube). 

\begin{assumption}[\textbf{Minkowski sum, \textit{optional}}]\label{The:MinkowskiSumAss}
  (a) Vertex enumeration for the sets $\BU$, $\BX$, $\BD$ is computationally prohibitive. (b) Computing the Minkowski sum of any two of these sets, or their linear transformations, is intractable.
\end{assumption}

\begin{assumption}[\textbf{Pontryagin difference}]\label{The:PontryaginDiffAss}
  It is possible to perform linear transformations and Pontryagin differences of all relevant sets.
\end{assumption}

To exemplify Assumption \ref{The:PontryaginDiffAss}, consider a \emph{linear transformation} of the set $\BD\subset\BR^{n}$ by a matrix $K\in\BR^{m\times n}$. The set $K\BD$ is a polyhedron in $\BR^{m}$, characterized by
\begin{equation}
    K\BD =\Bigl\{Kd\;\Big|\;d\in\BD\Bigr\}\ef
\end{equation}

\begin{proposition}\label{The:PontryaginDiff}
	Let $K\in\BR^{m\times n}$ be any linear transformation matrix. The Pontryagin difference $\BU\ominus K\BD$ is given by
	\begin{subequations}
	\begin{equation}\label{Equ:PontryaginDiff1}
	    u\in\BU\ominus K\BD\quad\Longleftrightarrow\quad G^{(\mathrm{u})}u\leq h^{(\mathrm{u})}-s
	\end{equation}
	where $s$ is a vector whose elements satisfy\\
	\begin{equation}\label{Equ:PontryaginDiff2}
	    \hspace*{4.5cm}s_{[i]}\eqdef \max_{d\in\BD} G^{(\mathrm{u})}_{[i,\cdot]}Kd\ef
	\end{equation}
	\end{subequations}
\end{proposition}

\begin{proof} Observe the following equivalence:
  \begin{align*}
    \BU\ominus K\BD &=\Bigl\{u\in\BR^{m}\;\Big|\;u+Kd\in\BU\;\;\fa d\in\BD\Bigr\}\\
                    &=\Bigl\{u\in\BR^{m}\;\Big|\;G^{(\mathrm{u})}\bigl(u+Kd\bigr)\leq h^{(\mathrm{u})}\;\;\fa d\in\BD\Bigr\}\\
                    &=\Bigl\{u\in\BR^{m}\;\Big|\;G^{(\mathrm{u})}u\leq h^{(\mathrm{u})}-G^{(\mathrm{u})}Kd\;\;\fa d\in\BD\Bigr\}\ef 
  \end{align*}
  The result follows immediately from the last line.
\end{proof}

The maximization on the right-hand side of \eqref{Equ:PontryaginDiff1} represents a Linear Program (LP). Hence, computing the Pontryagin difference in Proposition \ref{The:PontryaginDiff} amounts to solving a finite number of LPs. In this exemplary case, the number of LPs corresponds to the size of the H-representation of $\BU$; each LP has $n$ decision variables and as many linear constraints as the H-representation of $\BD$. The complexity of the resulting polyhedron $\BU\ominus K\BD$ is \emph{exactly} the same as that of the original polyhedron $\BU$, as only the right-hand side of the H-representation is modified.

Almost all Robust MPC approaches (except the one used in this paper) require the pre-computation of robust positively invariant sets \cite{Mayne:2005,Goulart:2006}. Their definition \cite{Blanchini:1999} is based on a \emph{pre-stabilization} of system \eqref{Equ:DTSystem}:
\begin{equation}
  u_{k}=Kx_{k}\quad\Longrightarrow\quad x_{k+1}=\bigl(A+BK)x_{k}+d_{k}\ec
\end{equation}
where $K\in\BR^{m\times n}$ is a stabilizing \emph{feedback gain}, which exists by virtue of Assumption \ref{The:System}.

\begin{definition}[(robust) positively invariant set]\label{Def:PosInvSet}
  (a) A set $\CX\subset\BX$ is called a \emph{positively invariant set} (PI set) of system \eqref{Equ:DTSystem} if for any $x\in\CX$ it holds that $Kx\in\BU$ and $(A+BK)x\in\CX$. (b) It is called a \emph{robust positively invariant set} (RPI set) if, additionally, $(A+BK)x+d\in\CX$ for all $d\in\BD$.
\end{definition}

The complexity of computing RPI sets can be easy, in particular cases, but it is generally very hard \cite{Kerrigan:2000,RakEtAl:2005,RakBar:2010}. It depends on whether a maximal, minimal, or any type of RPI set is required or acceptable; and it increases quickly with the system dimensions (`curse of dimensionality'). Practical experience indicates that a computation is almost always intractable for systems with state dimension $n\geq 5$. In this paper, simply the following is assumed.

\begin{assumption}[\textbf{robust positively invariant set, \textit{optional}}]\label{The:InvariantSet}
  The computation of an RPI set for system \eqref{Equ:DTSystem} is prohibitive.
\end{assumption}

%\begin{definition}[\textbf{large-scale system}]\label{The:LargeScale}
%  System \eqref{Equ:DTSystem} is of \emph{large scale}, in the sense of this paper, if it satisfies all positive Assumptions \ref{The:System}, \ref{The:Constraints}, \ref{The:NominalMPC}, \ref{The:PontryaginDiffAss} and at least one negative Assumption \ref{The:MinkowksiSumDef} or \ref{Def:PosInvSet}.
%\end{definition}

\section{DEADBEAT ROBUST MODEL PREDICTIVE CONTROL}

\subsection{Deadbeat Prestabilization}\label{Sec:DeadbeatPolicy}

A slightly modified version of the original DRMPC approach \cite{Schildbach:2025} is used in this paper. It is based on a \emph{deadbeat (disturbance feedback) policy} for pre-stabilization, instead of \emph{deadbeat input sequences} \cite{SchildbachAbbas:2025}. For the following derivation, consider system \eqref{Equ:DTSystem}, starting from an initial condition $x_{0}=d_{0}$ that corresponds to a single disturbance $d_{0}\in\mathbb{D}$.

\begin{definition}[\textbf{deadbeat horizon}]\label{The:DeadbeatHorizon}
The \emph{deadbeat horizon} $M$ \cite{Schildbach:2025} is the smallest number such that the matrix $\mathcal{P}_{M}\in\BR^{n\times nm}$,
\begin{equation}\label{Equ:ControllabilityMatrix}
  \mathcal{P}_{M}\eqdef
  \Bigl[
  \begin{array}{c|c|c|c|c}
     A^{M-1}B & A^{M-2}B_{0} & \;\cdots\; & AB & B
  \end{array}
  \Bigr]\ec
\end{equation}
has full rank $n$.
\end{definition}
By virtue of Assumption \ref{The:System}(b), $M$ is well defined and satisfies $M\leq n$, because $\mathcal{P}_{n}$ is the \emph{controllability matrix} of \eqref{Equ:DTSystem}. This implies the following result.

\begin{lemma}[\textbf{deadbeat disturbance feedback policy}]\label{The:DeadbeatPolicy}
There exists a \emph{deadbeat (disturbance feedback) policy}, consisting of $M$ linear feedback gains $K_{0},K_{1},\dots,K_{M-1}\in\BR^{m\times n}$ with
\begin{equation}\label{Equ:DeadbeatPolicy}
  \bar{u}_{0}=K_{0}d_{0}\ec\;\bar{u}_{1}=K_{1}d_{0}\ec\;\dots\ec\;\bar{u}_{M-1}=K_{M-1}d_{0}\ec
\end{equation}
such that $x_{M}=0$ for any initial condition $x_{0}=d_{0}$.
\end{lemma}

\begin{proof}
First, observe that the first $M$ states of the nominal system---i.e., system \eqref{Equ:DTSystem} without any disturbances---can be expressed as
\begin{subequations}\label{Equ:FirstMStates}\begin{align}
  \bar{x}_{1} &= Ad_{0}+Bu_{0}\ec\\
  \bar{x}_{2} &= A^{2}d_{0}+ABu_{0}+Bu_{1}\ec\\
  &\cdots\nonumber\\
  \bar{x}_{M-1} &= A^{M-1}d_{0}+A^{M-2}Bu_{0}+\dots+Bu_{M-2}\ec\\
  \bar{x}_{M} &= A^{M}d_{0}+A^{M-1}Bu_{0}+\dots+Bu_{M-1}\;.
\end{align}\end{subequations}
Since the matrix $\mathcal{P}_{M}$ in \eqref{Equ:ControllabilityMatrix} has full rank, there exists a sequence of deadbeat inputs $\bar{u}_{0},\bar{u}_{1},\dots,\bar{u}_{M-1}$ that force $\bar{x}_{M}=0$. Using $\bar{x}_{M}\eqcon 0$, (\ref{Equ:FirstMStates}d) becomes
\begin{equation}\label{Equ:DeadbeatCondition1}
  -A^{M}d_{0} \eqcon A^{M-1}B\bar{u}_{0}+A^{M-2}B\bar{u}_{1}+\dots+B\bar{u}_{M-1}\;.
\end{equation}
To verify that the deadbeat policy is indeed well-defined, substitute \eqref{Equ:DeadbeatPolicy} into \eqref{Equ:DeadbeatCondition1} to obtain
\begin{multline}\label{Equ:DeadbeatCondition2}
  -A^{M}d_{0} =\\ 
  \Bigl[
  \begin{array}{c|c|c|c}
  A^{M-1}BK_{0} 
  & A^{M-2}BK_{1}
  &\dots
  &BK_{M-1}
  \end{array}
  \Bigr]d_{0}\;.
\end{multline}
Since the deadbeat policy must work for all $d_{0}\in\BR^{n}$, the matrices multiplying $d_{0}$ in \eqref{Equ:DeadbeatCondition2} on both sides must be equal, thus
\begin{multline}\label{Equ:DeadbeatCondition3}
  -A^{M} = \\
  \Bigl[
  \begin{array}{c|c|c|c}
  A^{M-1}B
  & A^{M-2}B
  &\dots
  &B
  \end{array}
  \Bigr]
  \left[
  \begin{array}{c}
  K_{0} \\\hline
  K_{1}\\\hline
  \dots\\\hline
  K_{M-2}\\\hline
  K_{M-1}
  \end{array}
  \right]\;.
\end{multline}
This can be expressed as a linear equation system to solve for the linear feeback gains $K_{0},K_{1},\dots,K_{M-1}$:
\begin{multline}\label{Equ:DeadbeatCondition4}
  -\vct{A^{M}} =\\
  \left[\hspace*{-0.2cm}
  \begin{array}{ccccc}
    \mathcal{P}_{M} & 0 & \cdots & 0 & 0\\
    0 & \mathcal{P}_{M} & \cdots & 0 & 0\\
    \vdots & & \ddots & & \vdots\\
    0 & 0 & \cdots & \mathcal{P}_{M} & 0\\
    0 & 0 & \cdots & 0 & \mathcal{P}_{M}
  \end{array}\hspace*{-0.2cm}\right]
  \vct{\left[
  \begin{array}{c}
  K_{0} \\\hline
  K_{1}\\\hline
  \dots\\\hline
  \hspace*{-0.2cm}K_{M-2}\hspace*{-0.2cm}\\\hline
  \hspace*{-0.2cm}K_{M-1}\hspace*{-0.2cm}
  \end{array}
  \right]}
  \;.
\end{multline}
The claim follows from the fact that $\mathcal{P}_{M}$ has full rank $n$, according to Definition \ref{The:DeadbeatHorizon}, and hence the matrix on the right-hand side of \eqref{Equ:DeadbeatCondition4} has full rank $Mn$.
\end{proof}

\begin{remark}
  Alternatively, it is possible to calculate a \emph{time-invariant} disturbance feedback gain $K\in\BR^{m}$ by pole placement. Placing all $n$ poles at $0$ by an appropriate choice of $K$ renders the matrix $(A+BK)$ nilpotent; hence the deadbeat property follows. However, this only works for the choice of $M=n$.
\end{remark}

Note that Lemma \ref{The:DeadbeatPolicy} not only proves the existence of the deadbeat policy. It also provides a constructive approach for computing the corresponding feedback gains $K_{0},K_{1},\dots,K_{M-1}$, namely by solving the linear equation system \eqref{Equ:DeadbeatCondition4}. In fact, \eqref{Equ:DeadbeatCondition4} can be decomposed into $n$ independent linear equation systems, each for one the $n$ collective columns of the feedback gains $K_{0},K_{1},\dots,K_{M-1}$. This poses no computational problem, even for large scale systems, and can be performed offline.

\begin{remark}\label{The:ChoiceOfM}
  (a) If the deadbeat policy is not unique, then some additional cost function (e.g., least squares) has to be defined when solving for $K_{0},K_{1},\dots,K_{M-1}$, subject to the equality constraint \eqref{Equ:DeadbeatCondition4}.
  (b) The deadbeat horizon $M$ may be deliberately chosen as a higher number than the smallest number for which $\mathcal{P}_{M}$ has full rank $n$. This opens (additional) degrees of freedom in the choice of the deadbeat policy.
  (c) The condition $N\geq n$ for the choice of the prediction horizon, as stated above, may be relaxed to $N\geq M$.
\end{remark}

\subsection{Deadbeat Robust Model Predictive Control}\label{Sec:FTOCP}

For the setup of the Finite-time Optimal Control Problem (FTOCP), it is assumed that a sequence of deadbeat feedback gains $K_{0},K_{1},\dots,K_{M-1}$ have been pre-computed offline, as described in Section \ref{Sec:DeadbeatPolicy}.

Similar to \cite{Schildbach:2025}, the setup of the DRMPC problem is described only for step $k=0$, as the generalization to subsequent steps $k=1,2,\dots$ is straightforward. Based on Lemma \ref{The:DeadbeatPolicy}, the following affine disturbance feedback policy should be applied in order to account for the deadbeat of each disturbance over the prediction horizon:
\begin{subequations}\label{Equ:DistCompensation}\begin{align}
  u_{0}&=u_{0|0}\ec\\
  u_{1}&=u_{1|0}+K_{0}d_{0}\ec\\
  u_{2}&=u_{2|0}+K_{1}d_{0}+K_{0}d_{1}\ec\\
  u_{3}&=u_{3|0}+K_{2}d_{0}+K_{1}d_{1}+K_{0}d_{2}\ec\\
  &\;\;\vdots\nonumber\\
  u_{M}&=u_{M|0}+\,\,K_{M-1}d_{0}+\hdots+K_{0}d_{M-1}\ec
\end{align}\end{subequations}
\noindent up to the deadbeat horizon, and beyond that
\addtocounter{equation}{-1}
\begin{subequations}\setcounter{equation}{5}\begin{equation}
  u_{j}=u_{j|0}+\sum_{i=0}^{M-1}K_{i}d_{j-i-1}\qquad\fa M\leq j\leq N-1\ef
\end{equation}\end{subequations}
The new notation $u_{k+j|k}$ is used for the nominal control input in step $k+j$, as planned in step $k$, and $x_{k+j|k}$ for the corresponding predicted nominal state, where $k,j\in\BZ_{0+}$. This scheme suggests the following tightening of the input constraints for the nominal control inputs:
\begin{subequations}\label{Equ:TightenedIC}\begin{align}
  u_{0|0}&\in\BU\ec\\
  u_{1|0}&\in\BU\ominus K_{0}\BD\ec\\
  u_{2|0}&\in\BU\ominus K_{0}\BD\ominus K_{1}\BD\ec\\
  u_{3|0}&\in\BU\ominus K_{0}\BD\ominus K_{1}\BD\ominus K_{2}\BD\ec\\
  &\;\;\vdots\nonumber\\
  u_{M|0}&\in\BU\ominus K_{0}\BD\ominus K_{1}\BD\ominus\hdots\ominus K_{M-1}\BD\ec
\end{align}\end{subequations}
\noindent up to the deadbeat horizon, and beyond that
\addtocounter{equation}{-1}
\begin{subequations}\setcounter{equation}{5}\begin{multline}
  u_{j|0}\in\BU\ominus K_{0}\BD\ominus K_{1}\BD\ominus\cdots\ominus K_{M-1}\BD\\\fa M\leq j\leq N-1\ef
\end{multline}\end{subequations}
For the tightening of the state constraints, define the auxiliary matrices
\begin{subequations}\label{Equ:AuxMatrices}\begin{align}
  \Phi_{0}&\eqdef A+BK_{0}\ec\\
  \Phi_{1}&\eqdef A^2+ABK_{0}+BK_{1}\ec\\
  \Phi_{2}&\eqdef A^3+A^2BK_{0}+ABK_{1}+BK_{2}\ec\\
  &\;\;\vdots\nonumber\\
%\Phi_{M-2}&\eqdef A^{M-1}+A^{M-2}BK_{0}+A^{M-3}BK_{1}+\dots+ABK_{M-3}+BK_{M-2}\ef
\Phi_{M-2}&\eqdef A^{M-1}+A^{M-2}BK_{0}+\dots+BK_{M-2}\ef
\end{align}\end{subequations}
Note that, by construction, the next auxiliary matrix $\Phi_{M-1}=0$, due to \eqref{Equ:DeadbeatCondition3}. With the help of these auxiliary matrices, the constraints for the nominal states should be tightened as follows:
\begin{subequations}\label{Equ:TightenedSC}\begin{align}
  x_{1|0}&\in\BX\ominus\BD\ec\\
  x_{2|0}&\in\BX\ominus\BD\ominus\Phi_{0}\BD\ec\\
  x_{3|0}&\in\BX\ominus\BD\ominus\Phi_{0}\BD\ominus\Phi_{1}\BD\ec\\
  &\;\;\vdots\nonumber\\
  x_{M|0}&\in\BX\ominus\BD\ominus\Phi_{0}\BD\ominus\Phi_{1}\BD\ominus\hdots\ominus\Phi_{M-2}\BD\ec
\end{align}\end{subequations}
\noindent up to the deadbeat horizon, and beyond that
\addtocounter{equation}{-1}
\begin{subequations}\setcounter{equation}{5}\begin{multline}
  x_{j|0}\in\BX\ominus\BD\ominus\Phi_{0}\BD\ominus\Phi_{1}\BD\ominus\hdots\ominus\Phi_{M-2}\BD\\\fa M\leq j\leq N\ef
\end{multline}\end{subequations}

The computation of all constraint sets in (\ref{Equ:TightenedIC}a,b,c,d,...,e) and (\ref{Equ:TightenedSC}a,b,c,d,...,e) is tractable, based on Assumption \ref{The:PontryaginDiffAss}. These computations can be performed \emph{offline}.

Combining these constraints with a cost function, based on a \emph{stage cost} $\ell:\BX\times\BU\to\BR$, yields the following \textbf{DRMPC (optimization) problem} to be solved \emph{online}, in each time step:
\begin{subequations}\label{Equ:DRMPCProblem}\begin{align}
  \min\enspace&\sum_{j=0}^{N-1}\ell\left(u_{j|0},x_{j|0}\right)\\
  \text{s.t.}\enspace& x_{j+1|0}=Ax_{j|0}+Bu_{j|0}\;\;\fa\,j=0,\hdots,N-1\,,\\
   &\quad x_{0|0}=x_{0}\ec\\
   &\quad\eqref{Equ:TightenedIC}\ec\quad\eqref{Equ:TightenedSC}\ec\\
   &\quad x_{N|0}=0\ef
\end{align}\end{subequations}

\begin{remark}[\textbf{computational complexity}]\label{The:Complexity}
  The DRMPC problem \eqref{Equ:DRMPCProblem} has \emph{exactly the same} computational complexity as the corresponding Nominal MPC problem. In particular, it is same type of optimization problem, with the same number of decision variables and the same number of constraints. Hence it is real-time solvable, based on Assumption \ref{The:NominalMPC}.
\end{remark}

\subsection{System Theoretic Guarantees}\label{Sec:SystemTheory}

Under the DRMPC regime, in each time step $k=0,1,\hdots$ the first element of the optimal input sequence from \eqref{Equ:DRMPCProblem}, $u_{k}=u_{k|k}^{\star}$, is applied to the system.

\begin{theorem}[\textbf{recursive feasibility}]\label{The:RecFeasibility}
  If the DRMPC problem \eqref{Equ:DRMPCProblem} is feasible at $k=0$ for $x_{0}$, it remains feasible for all future states $x_{1},x_{2},\hdots$ of the closed-loop system \eqref{Equ:DTSystem} under the DRMPC regime, for any admissible sequence of disturbances.
\end{theorem}

\begin{proof}
  The proof is analogous to \cite[Thm.\,1]{Schildbach:2025}.
\end{proof}

Let $\CX_{N}\subseteq\BX$ be the \emph{feasible set}, i.e., the set of initial conditions $x_{0}$ for which problem \eqref{Equ:DRMPCProblem} is feasible. The \emph{state feedback control law} $\kappa_{N}:\CX_{N}\to\BU$ of the DRMPC regime is defined as $\kappa_{N}(x_{k})=u_{k|k}^{\star}$. According to Theorem \ref{The:RecFeasibility}, $\CX_{N}$ is an RPI set on which $\kappa_{N}$ is well-defined.

Due to the presence of disturbances, the appropriate stability concept to be used is \emph{input-to-state stability (ISS)} \cite{LimonEtAl:2009,RaMaDi:2018}.

\begin{definition}[\textbf{input-to-state stability}]\label{Def:Stability}
    System \eqref{Equ:DTSystem} under any state feedback law $\kappa:\CP\to\BU$ is called \textbf{input-to-state stable (ISS)} on an RPI set $\CP\subseteq\BX$ if there exist a $\text{KL}$-function $\beta$ and a $\text{K}$-function $\gamma$ such that
    \begin{equation*}
        \|x_{k}\|\leq\beta\bigl(\|x_{0}\|,k\bigr)+\gamma\bigl(\max_{j<k}\|d_{j}\|\bigr) \quad\fa k\in\BZ_{0+}\;,
    \end{equation*}
    for any $x_{0}\in\CP$ and any disturbances $d_{0},d_{1},\hdots\in\BD$.
\end{definition}

ISS can be established based on a standard assumption on the stage cost function \cite{LimonEtAl:2009,RaMaDi:2018}.

\begin{assumption}[\textbf{stage cost}]\label{Ass:StageCost}
  The stage cost function $\ell:\BX\times\BU\rightarrow\BR_{0+}$ is chosen such that there exist two $\textrm{K}_{\infty}$-functions $\varepsilon_{1}$ and $\varepsilon_{2}$ with
  \begin{equation}\label{Equ:LyaFunction1}
    \varepsilon_{1}(\|x\|)\leq \ell\left(x,u\right)\leq \varepsilon_{2}(\|x\|)\quad\fa x\in\BX\;,\;u\in\BU\;.
  \end{equation}
\end{assumption}

\begin{theorem}[\textbf{input-to-state stability}]\label{The:Stability}
  The closed-loop system \eqref{Equ:DTSystem} under the DRMPC regime $\kappa_{N}:\CX_{N}\to\BU$ is ISS on $\CX_{N}$.
\end{theorem}

\begin{proof}
  The proof is based on the definition of an ISS Lyapunov function \cite{JiangWang:2001} and analogous to \cite[Thm.\,2]{Schildbach:2025}.
\end{proof}

\section{NUMERICAL STUDY}

DRMPC has previously been benchmarked against other standard methods of Robust MPC for small scale systems. Its performance was found to be similar to Tube MPC \cite{Mayne:2005,RakEtAl:2005} and MPC with Affine Disturbance Feedback \cite{Goulart:2006}, in terms of the size of its region of attraction \cite{Schildbach:2025}. Due the restrictiveness of the origin as the terminal set, DRMPC performs slightly worse when compared on the basis of same prediction horizon \cite{SchildbachAbbas:2025}. However, DRMPC is computationally simpler, so it performs slightly better when compared on the basis of identical (online) computation time.

Hence, instead of presenting a performance comparison, the goal of this section is to explore the boundaries of DRMPC in the context of large-scale systems. To this end, random linear systems \eqref{Equ:DTSystem} are generated,
\begin{equation}
  A=I_{n}+0.01\,\tilde{A}\qquad\text{and}\qquad B=\tilde{B}\ef
\end{equation}
The elements of $\tilde{A}\in\BR^{n\times n}$ and $\tilde{B}\in\BR^{n\times m}$ are randomly drawn from a standard normal distribution $\mathcal{N}(0,1)$, such that Assumption \ref{The:System}(b) is satisfied. The system dimensions $n$ and $m$ are variables in the experiments.

Similarly, the constraint sets $\BU$, $\BX$ and the disturbance set $\BD$ are randomly generated in the form of hyper-cubes. In detail,
\begin{subequations}
\begin{equation}
  G^{(\mathrm{u})}=
  \begin{bmatrix}
    +I_{m}\\
    -I_{m}
  \end{bmatrix}
  \,,\; G^{(\mathrm{x})}=
  \begin{bmatrix}
    +I_{n}\\
    -I_{n}
  \end{bmatrix}\,,\; G^{(\mathrm{d})}=
  \begin{bmatrix}
    +I_{n}\\
    -I_{n}
  \end{bmatrix}
\end{equation}
and the right-hand sides
\begin{equation}
  h^{(\mathrm{u})}=10\,\tilde{h}^{(\mathrm{u})}\,,\;h^{(\mathrm{x})}=100\,\tilde{h}^{(\mathrm{x})}\,,\;h^{(\mathrm{d})}=0.1\,\tilde{h}^{(\mathrm{d})}
\end{equation}
\end{subequations}
are randomly generated such that the elements of $\tilde{h}^{(\mathrm{u})},\tilde{h}^{(\mathrm{x})},\tilde{h}^{(\mathrm{d})}$ are drawn from a standard normal distribution $\mathcal{N}(0,1)$.

Since the online complexity of DRMPC is the same as for Nominal MPC, see Remark \ref{The:Complexity}, this analysis considers only the \emph{setup time}, i.e., the computation time for setting up the DRMPC problem (offline). This includes the computation of the deadbeat feedback gains \eqref{Equ:DeadbeatPolicy}, and the tightening of the input and state constraints \eqref{Equ:TightenedIC} and \eqref{Equ:TightenedSC}. 

All computations are performed on a MacBook Pro, with Apple M2 Max chip with 64 GB memory, running MatLab\textregistered\; 2024b. They only serve as a rough indication of the computational feasibility of DRMPC.

\renewcommand{\arraystretch}{1.41}
\begin{table}[h]
	\centering
	\begin{tabular}{|l|r||r|r|r|r|r|r|}
		\hline
		\multicolumn{2}{|c||}{$\,$} & \multicolumn{6}{c|}{\textbf{state dimension ($n$)}}\\\cline{3-8}
		\multicolumn{2}{|c||}{$\,$} & \multicolumn{1}{c|}{$10$} & \multicolumn{1}{c|}{$60$} & \multicolumn{1}{c|}{$120$} & \multicolumn{1}{c|}{$600$} & \multicolumn{1}{c|}{$900$} & \multicolumn{1}{c|}{$1200$}\\\hline\hline
		\multirow{4}{0.3cm}{\rotatebox{90}{\textbf{input dimension ($m$)\;}}}
		 	& $5$  &        0.14\,s &       1.16\,s &       8.45\,s &    565\,s & $\star\;\;\;$ & $\star\;\;\;$ \\\cline{2-8}
			& $10$ &        0.12\,s &       1.08\,s &       4.32\,s &    246\,s &        863\,s &       2065\,s \\\cline{2-8}
			& $30$ &  $\circ\;\;\;$ &       0.48\,s &       1.64\,s &    119\,s &        317\,s &        680\,s \\\cline{2-8}
			& $60$ &  $\circ\;\;\;$ &       0.32\,s &       1.01\,s &    66.5\,s &       181\,s &        398\,s \\\cline{2-8}
			& $120$ & $\circ\;\;\;$ & $\circ\;\;\;$ &       0.67\,s &    40.2\,s &       131\,s &        233\,s \\\cline{2-8}
			& $300$ & $\circ\;\;\;$ & $\circ\;\;\;$ & $\circ\;\;\;$ &    29.0\,s &      67.7\,s &        211\,s \\\hline
	\end{tabular}
	\caption{Total computation time required for setting up the DRMPC problem (offline). Symbols: `$\circ$' means `not applicable,' `$\star$' means `no result due to ill-conditioning.' All numbers are averaged over 100 simulation runs. \label{Tab:Results}}
\end{table}
\renewcommand{\arraystretch}{1.0}

Table \ref{Tab:Results} has been created based on a minimal selection of the deadbeat horizon, i.e., $M=\frac{m}{n}$. Only cases where $m\leq n$ have been considered. Cases where $m>n$ are marked with `$\circ$.' In cases where $m\ll n$, the problem of solving \eqref{Equ:DeadbeatCondition4} for finding the deadbeat feedback gains becomes ill-conditioned. These cases are marked with `$\star$.'

The setup times in Table \ref{Tab:Results} indicate that DRMPC is a viable approach for system dimensions that are well beyond the capabilities of Tube MPC or Robust MPC with Affine Disturbance Feedback. For a randomly generated, dense linear system, the main limitation turns out to be the conditioning of the linear equation system \eqref{Equ:DeadbeatCondition4}. The author suspects that this will not be a problem for most well-designed (i.e., properly controllable) large-scale systems in practice.

\section{CONCLUSION}

DRMPC shows a competitive performance for small-scale systems, compared to other state-of-the-art Robust MPC approaches. Moreover, DRMPC pushes the boundaries of Robust MPC to systems of large scale, i.e., with several hundred states and inputs. It has exactly the same (online) computational complexity as a Nominal MPC approach. The computational complexity of the (offline) setup of DRMPC is reduced because it does not require the computation of an RPI set, which suffers from the curse of dimensionality. This has been validated by a small numerical study.

% --------------------------------------------------------------------------------------------------

\bibliographystyle{IEEEtran}
\bibliography{bibmath,bibcontr,bibeng,bibecon}

% -------------------------------------------------------------------------------------------------

\end{document}